\documentclass[12pt,english]{article}
\usepackage[T1]{fontenc}
\usepackage[utf8]{inputenc}
\usepackage{geometry}
\usepackage{babel}
\usepackage{mathtools}
\usepackage{amsthm}
\usepackage{amssymb}
\usepackage[unicode=true,pdfusetitle,
 bookmarks=true,bookmarksnumbered=false,bookmarksopen=false,
 breaklinks=false,backref=false,hidelinks]
 {hyperref}

\makeatletter
\numberwithin{equation}{section}
\numberwithin{figure}{section}
\theoremstyle{plain}
\newtheorem{thm}{Theorem}
\newtheorem{prop}[thm]{Proposition}
\newtheorem{lem}[thm]{Lemma}

\theoremstyle{definition}
\newtheorem{defn}[thm]{Definition}

\theoremstyle{remark}
\newtheorem{rem}[thm]{Remark}
\theoremstyle{plain}

\date{}
\usepackage{tikz-cd}
\IfFileExists{libertine.sty}{\usepackage[tt=false]{libertine}}{}
\usepackage[parfill]{parskip}
\def\thm@space@setup{%
  \thm@preskip=10pt plus 2pt minus 2pt
  \thm@postskip=10pt plus 2pt minus 2pt
}
\usepackage{enumitem}
\setlist[itemize]{noitemsep,topsep=5pt}

\usepackage{titlesec}
\usepackage{titletoc}
\titleformat{\section}{\large\bfseries\filleft}{\thesection}{1em}{}[{\titlerule[0.8pt]}]
\titleformat{\part}[block]
  {\normalfont\Large\bfseries\raggedright}
  {}{0pt}{}
\titlespacing*{\part}{0pt}{3.5ex plus 1ex minus .2ex}{1.5ex}

\titlecontents{part}
  [0pt]
  {\addvspace{.9em}\bfseries}
  {}
  {}
  {}
\titlecontents{section}
  [1.5em]
  {\setlength{\parfillskip}{0pt plus 1fil}\contentsmargin{0pt}}
  {\contentslabel{1.75em}}
  {}
  {\titlerule*[.6pc]{.}\hspace*{1.2em}\makebox[0pt][r]{\thecontentspage}}

\newenvironment{ack}{\textit{Acknowledgements.}}{}
\renewcommand\labelenumi{(\roman{enumi})}
\renewcommand\theenumi\labelenumi

\DeclareMathOperator{\codim}{codim}
\DeclareMathOperator{\ch}{char}

\DeclareMathOperator{\Fr}{Fr}

\DeclareMathOperator{\Sym}{Sym}

\DeclareMathOperator{\im}{im}

\DeclareMathOperator{\height}{ht}

\DeclareMathOperator{\Id}{Id}

\DeclareMathOperator{\init}{in}
\DeclareMathOperator{\Sing}{Sing}
\DeclareMathOperator{\Jac}{Jac}
\DeclareMathOperator{\Tr}{Tr}

\makeatother

\global\long\def\A{\mathbb{A}}%
\global\long\def\C{\mathbb{C}}%
\global\long\def\F{\mathbb{F}}%
\global\long\def\P{\mathbb{P}}%
\global\long\def\Q{\mathbb{Q}}%
\global\long\def\R{\mathbb{R}}%
\global\long\def\O{\mathcal{O}}%
\global\long\def\Z{\mathbb{Z}}%
\global\long\def\ol#1{\overline{#1}}%

\begin{document}

\title{Birch's theorem over function fields with quadratically many variables}
\author{Matthew Hase-Liu}
\maketitle
\begin{abstract}
For smooth hypersurfaces over rational function fields of characteristic greater than the degree and with sufficiently large
constant field, we improve the number of variables required in Birch's theorem from an exponential function of the degree to a quadratic one. This agrees, up to constants, with the sharp quadratic threshold for the unconditional existence of rational points on smooth hypersurfaces. 

A circle method argument reduces the required cancellation to lower bounds for the codimensions of certain singular loci associated with complete exponential sums over finite fields. Our main innovation is a new method for proving these bounds: we introduce the notion of multiplication rank for the linear functionals indexing these exponential sums and combine the resulting rank stratification with a weighted degeneration of the Jacobian equations to obtain a codimension estimate that grows linearly with multiplication rank.
\end{abstract}
\tableofcontents{}
\section{Introduction}
Birch famously showed in \cite{birch1962forms} that a smooth degree $d$ hypersurface over the rational numbers satisfies the expected asymptotic formula for its integral points provided that the number of variables exceeds $(d-1)2^d$. Lee \cite{lee2011birch} subsequently established a function field analogue with the same exponential dependence on the degree and an explicit error term.

In his thesis \cite{pugin2011}, Pugin introduced an alternative to Birch's original approach in which the minor arcs are treated as complete exponential sums over finite fields and estimates of Katz reduce the required cancellation to bounds on the dimensions of certain singular loci. Sawin \cite{sawin2024asymptoticwaringsproblemfunction} then developed this method further by bounding the tangent spaces of these singular loci to obtain improved results for Waring's problem over function fields. In this paper, we introduce a different extension of Pugin's approach to address arbitrary smooth hypersurfaces and reduce the required number of variables from exponential to quadratic in the degree.

The quadratic growth in our hypothesis is unavoidable for a result guaranteeing rational points on every smooth degree $d$ hypersurface without local-solubility assumptions. Indeed, in Lemma \ref{lem:quad-sharp}, we show the existence of a smooth degree $d$ hypersurface in $\P^{d^2-1}_{\F_q(t)}$ with no rational point. Together with Lang's $C_2$ theorem, this shows that the threshold for unconditional existence of rational points, even among smooth hypersurfaces, has quadratic order. However, this sharpness statement concerns unconditional existence; it does not determine the optimal range for a Birch asymptotic with positive leading constant under local-solubility hypotheses.

We now explain the setup in more detail. Let $K=\F_{q}(t),p=\ch\F_{q},d\ge2,$ and $p>d$. Fix $F\in K[x_{1},\ldots,x_{n}]$
a homogeneous form of degree $d$. After clearing denominators and ensuring that the coefficients have no non-constant common factors,
we may assume that it lives in $\F_{q}[t][x_{1},\ldots,x_{n}]$. Let
$c$ be the maximum of the degrees in $t$ of the coefficients of $F$. 

Write $P_{i}=H^{0}(\P^{1},\O(i\infty))$, which we identify with polynomials in $\F_{q}[t]$ with $\deg\le i$, and let\[
N_{e}(F)=\#\{\vec{g}\in P_{e}^{n}\colon F(\vec{g}(t))=0\}.
\]Note that $F(\vec{g}(t))\in P_{de+c}$. 

Moreover, the closed points of $\P^{1}$ are naturally in bijection with
the places of the function field $\F_{q}(\P^{1})$, and we denote
both by $|\P^{1}|$. For $v$ a closed point of $\P^{1}$, let $\O_{\P^{1},v}$
be the corresponding local ring with uniformizer $\pi_{v}$ and $F_{v}(\vec{x})$
be the local equation induced by $F$---more precisely, for $v=\infty$,
take $u=t^{-1}$ and set $F_{\infty}(u,\vec{x})=u^{c}F(u^{-1},\vec{x})$.
Then, denote by 
\[
\ell_{v}(F)=\lim_{m\to\infty}q^{-m(n-1)\deg v}\#\{\vec{x}\bmod\pi_{v}^{m}\colon F_{v}(\vec{x})\equiv0\bmod\pi_{v}^{m}\}.
\]
The expected asymptotic formula for the integral points is the following:
\begin{equation}\label{eq:mainasy}
N_{e}(F)=q^{n(e+1)-(de+c+1)}\left(\prod_{v\in|\P^{1}|}\ell_{v}(F)+O_{F,q,d,n}(q^{-\eta e})\right)\text{ as }e\to\infty.
\end{equation}
The main result establishes the asymptotic in quadratically many variables.
\begin{thm}
\label{thm:main}Assume that $V(F)\subset\P_{K}^{n-1}$ is smooth. Then, we have \eqref{eq:mainasy} as soon as  
\[
\eta=\frac{n}{d-1}\left(\frac{1}{2}-\log_q(d-1)\right)-d>0.
\] Specifically, the local density limits exist and the Euler product converges absolutely.
\end{thm}

The dependence on the constant field is explicit: provided that
$n>2d(d-1)$, the condition $\eta>0$ is equivalent to
\[
q>(d-1)^{\frac{2n}{n-2d(d-1)}}.
\] The logarithmic term in $\eta$ comes from the factor $(d-1)^{\codim \Sing(P)}$ in the complete sum estimate below---this is essentially Katz's bounds for vanishing cohomology combined with Wan and Zhang's degree-wise Betti bounds. For example, if $n\ge 4d(d-1)$, then
$q>(d-1)^4$ suffices.

Next, we briefly describe the structure of the proof, which follows the usual shape of the circle method. Orthogonality of characters expresses the normalized counting function as a sum of complete exponential sums $S(\alpha)$, indexed by linear functionals $\alpha$ on $P_{de+c}$. We define $\deg(\alpha)$ to be the minimal length of a finite subscheme of $\P^1$ through which $\alpha$ factors, and introduce a new, related quantity $r(\alpha)$. In particular, we will see in Lemma \ref{lem:rankvsdeg} that\[r(\alpha)=\min(\deg(\alpha), e+1).\] The major arcs are the functionals satisfying $\deg(\alpha)\le e$, and the remaining functionals form the minor arcs. 

The main new ingredient is a lower bound for the codimension of the singular locus of $P_{\alpha,F}(\vec{g})=\alpha(F(\vec{g}))$. Briefly, we translate the parameter so that a suitable constant coefficient specialization of $F$ is smooth, and then we apply row reduction and a weighted degeneration of the Jacobian ideal to establish this bound.

Our complete exponential sum estimate below then gives the required cancellation on the minor arcs. On the major arcs, the minimal supporting subschemes allow us to reindex the character sums as a truncation of the singular series. The same estimate provides absolute convergence, while local Fourier orthogonality combined with the Chinese remainder theorem identifies the singular series with the product of local densities.

Finally, we explain the sharpness of the quadratic dependence by showing the existence of a smooth hypersurface of degree $d$ in $d^2$ variables with no rational point. The proof below combines the classical norm form construction showing the sharpness of the $C_2$ bound with a perturbation argument ensuring
smoothness. The latter argument was suggested by ChatGPT 5.6 Sol.

\begin{lem}\label{lem:quad-sharp} Let $d\ge 2$ and suppose $\ch(\F_q)\nmid d$. Then, there is a smooth degree $d$ hypersurface $X\subset \P_{\F_q(t)}^{d^2-1}$ with no $\F_q(t)$-point.
\end{lem}
\begin{proof}
    Let $K=\F_q(t)$ and $K_v=\F_q(\!(t)\!)$. Also let $N$ be the norm form of $\F_{q^d}/\F_q$. Divide the $d^2$ variables into $d$ blocks $\vec x_0,\ldots, \vec x_{d-1}$, and set \[F_0 = \sum_{j=0}^{d-1}t^jN(\vec x_j).\] Observe that this form is anisotropic over $K_v$ since the valuations of its non-zero summands are pairwise distinct mod $d$: Choose an $\F_q$-basis of $\F_{q^d}$. After base change to $K_v$, the norm form becomes the norm for the unramified degree $d$ extension $\F_{q^d}(\!(t)\!)/K_v.$ Thus, if $\vec x_j\ne 0$, then $v_t(N(\vec x_j))=d\,v_L(\vec x_j),$ so $v_t(t^jN(\vec x_j))= j\bmod d.$ The non-zero summands of $F_0$ therefore have pairwise distinct valuations, so the one with smallest valuation cannot cancel with the others. 

    Now, consider the vector space of degree $d$ forms in $d^2$ variables over $K_v$. Equip it with the $t$-adic topology. Then the anisotropic forms comprise a $t$-adically open subset: $|F_0|$ has a positive minimum on the compact set $\{\vec x\in \F_q[\![t]\!]^{d^2}\colon \max_i |x_i|=1\}$, so any other degree $d$ form whose coefficients are sufficiently close to those of $F_0$ will remain anisotropic by the non-archimedean triangle inequality.

    On the other hand, smooth forms constitute a non-empty Zariski open subset of the same space of forms because the Fermat hypersurface is smooth when $\ch(\F_q)\nmid d$. Since $K_v$ is infinite, this subset is $t$-adically dense. Then, we may choose a smooth form $G$ in the open neighborhood of $F_0$ consisting of anisotropic forms constructed above. Both smoothness and anisotropy are $t$-adically open conditions, so there is a $t$-adic neighborhood of $G$ consisting of forms having both properties. Since $K$ is dense in $K_v$, this neighborhood contains a form $F$ defined over $K$. Now, let $X$ be cut out by $F$, and the desired claim follows.
\end{proof}
\begin{ack}
    I'd like to thank Tim Browning, Jordan Ellenberg, Jakob Glas, Will Sawin, and Victor Wang for their interest and helpful comments. I'm especially grateful to Will for suggesting that the distribution of pivot degrees could be used to sharpen the numerical constant in the main theorem, to Jakob for sharing with me some very enlightening follow-up applications, and to Tim and Victor for helping me contextualize the results properly (see Remark \ref{rem:context} for example).
    
    To be transparent about AI use, ChatGPT 5.6 Sol was used to proofread and streamline many of the arguments and to suggest simplifications in the proofs of Lemma \ref{lem:quad-sharp}, Proposition \ref{prop:codimbound}, and Section \ref{sec: compute_sing_series}. The main new ideas of the paper, including the use of multiplication rank to control the relevant singular locus and the weighted degeneration of the Jacobian equations, were developed by the author before this AI-assisted revision. All AI-assisted suggestions were independently checked by the author.
\end{ack}
\section{Major/minor arcs and rank}
In this section, we introduce the multiplication rank of a functional, an improved version of Katz's complete exponential bounds, and the major/minor arc division. 

Fix a non-trivial additive character $\psi\colon\F_{q}\to\C^{\times}$.
Let $\alpha\in P_{de+c}^{\vee}$ and define 
\[
S(\alpha)=\sum_{\vec{g}\in P_{e}^{n}}\psi(\alpha(F(\vec{g}))).
\]
By orthogonality of characters, we have 
\[
N_{e}(F)=q^{-(de+c+1)}\sum_{\alpha}S(\alpha).
\]
We say $\alpha\sim Z$ for $Z\subset\P^{1}$ a finite subscheme if
$\alpha$ factors through the restriction map $P_{de+c}\to H^{0}(Z,\O_{Z}((de+c)\infty))$.
Define $\deg(\alpha)$ to be the minimum over all $\deg(Z)$ for which
$\alpha\sim Z$. See \cite{haseliu2024higher} for more details on this and the connection with Dirichlet approximation.

Call a functional $\beta\colon H^{0}(Z,\O_{Z}((de+c)\infty))\to\F_{q}$
\emph{primitive on $Z$} if it does not factor through a proper closed
subscheme of $Z$.

\begin{lem}
\label{lem:minimalsupport}Suppose that $\alpha\in P_{de+c}^{\vee}$
factors through finite subschemes $Z_{1},Z_{2}\subset\P^{1}$ of degree
at most $e$. Then $\alpha$ factors through $Z_{1}\cap Z_{2}$.
Consequently, every $\alpha$ with $\deg(\alpha)\le e$ has a unique
minimal supporting subscheme, and composition with the restriction map
gives a bijection
\[
\{\alpha\in P_{de+c}^{\vee}\colon\deg(\alpha)\le e\}
\longleftrightarrow
\{(Z,\beta)\colon\deg Z\le e,\ \beta\text{ primitive on }Z\}.
\]
\begin{proof}
This follows from the arguments of
\cite[Lemma 13]{haseliu2024higher} and \cite[Remark 14]{haseliu2024higher}.
\end{proof}
\end{lem}

We call the following quantity the \emph{multiplication rank} of $\alpha$:
\begin{defn}
Let $C_{\alpha}\colon P_{e}\to P_{(d-1)e+c}^{\vee}$ be the map $C_{\alpha}(h)(u)=\alpha(hu)$
and $r(\alpha)$ the rank of $C_{\alpha}$ (as a linear map). 
\end{defn}

\begin{lem}
\label{lem:rankvsdeg}For $\alpha\in P_{de+c}^{\vee}$, we have $r(\alpha)=\min(\deg(\alpha),e+1)$. 
\begin{proof}
For $\deg(\alpha)\le e+1$: Let $\alpha\sim Z$ with $\deg(Z)=\deg(\alpha)$.
Let $\ol{\alpha}$ be the induced functional $H^{0}(Z,\O_{Z}((de+c)\infty))\to\F_{q}$
($\alpha$ is $\ol{\alpha}$ composed with $P_{de+c}\to H^{0}(Z,\O_{Z}((de+c)\infty))$).

Then, consider the surjective maps $P_{e}\to H^{0}(Z,\O_Z(e\infty))$
and $P_{(d-1)e+c}\to H^{0}(Z,\O_Z(((d-1)e+c)\infty))$. The pairing
$H^{0}(Z,\O_{Z}(e\infty))\times H^{0}(Z,\O_{Z}(((d-1)e+c)\infty))\to\F_{q}$
induced by $\ol{\alpha}$ induces a map $H^{0}(Z,\O_{Z}(e\infty))\to H^{0}(Z,\O_{Z}(((d-1)e+c)\infty))^{\vee}$. 

Then $C_{\alpha}\colon P_{e}\to P_{(d-1)e+c}^{\vee}$ can be written
as 
\[
P_{e}\to H^{0}(Z,\O_{Z}(e\infty))\to H^{0}(Z,\O_{Z}(((d-1)e+c)\infty))^{\vee}\to P_{(d-1)e+c}^{\vee},
\]
and clearly it suffices to show the middle map has rank $\deg(\alpha)$. 

To see this, let $A=H^{0}(Z,\O_{Z})$ so that the middle pairing can
be expressed as $A\times A\to\F_{q}$, where $(a,a')$ is mapped to
$\lambda(aa')$ for some functional $\lambda$. Suppose there is some
non-zero $a\in A$ so that $\lambda(aa')=0$ for all $a'$. Then,
$\lambda$ kills the ideal $(a)$ so that it factors through $A/(a)$,
i.e. $\alpha$ factors through the smaller closed subscheme defined
by $a$ inside $Z$, which is a contradiction. That means $A\to A^{\vee}$
is injective, so it has full rank.

For $\deg(\alpha)>e+1$: Suppose that $\ker C_{\alpha}$ is non-zero
and take a non-zero $h$ in it. Viewed as a section of $\O(e\infty)$,
$h$ has a zero divisor $V(h)$ of length $e$. The exact sequence
$0\to H^{0}(\P^{1},\O((de+c)\infty-V(h)))\to H^{0}(\P^{1},\O((de+c)\infty))\to H^{0}(V(h),\O((de+c)\infty))\to0$
implies that $\alpha$ factors through $V(h)$, which means $\deg(\alpha)\le e$,
which is a contradiction. So $\ker C_{\alpha}$ is zero, which means
$r(\alpha)=e+1$. 
\end{proof}
\end{lem}

We define the major arcs to be $\deg\alpha\le e$ (instead of $e+1$ as done in \cite{haseliu2024higher}).
Then, we decompose our normalized counting function as follows:
\[
q^{(de+c+1)-n(e+1)}N_{e}(F)=q^{-n(e+1)}\sum_{\deg\alpha\le e}S(\alpha)+q^{-n(e+1)}\sum_{\deg\alpha>e}S(\alpha).
\]

\begin{rem}\label{rem:context}
To relate the preceding definition of the major and minor arcs to the
usual rational approximation viewpoint in the classical circle method,
we explain how $\deg(\alpha)$ simultaneously measures the denominator
and the error of a rational approximation.

Let $D=de+c$ and associate to $\alpha\in P_D^\vee$ the truncated
Laurent series
\[\theta_\alpha=\sum_{j=0}^{D}\alpha(t^j)t^{-j-1},\]
where $\F_q(\!(t^{-1})\!)$ is equipped with the absolute value determined
by $|t|=q$. A finite subscheme $Z\subset\P^1$ has an affine part cut out by a monic polynomial $Q$ and a part of some length $m$ supported
at $\infty$. If $R=\deg Q$ and $s=R+m$, then the space of sections in $P_D$ vanishing on $Z$ is $QP_{D-s}$.

If $A$ is the polynomial part of $Q\theta_\alpha$, the coefficient of
$t^{-j-1}$ in $Q\theta_\alpha-A$ is $\alpha(Qt^j)$. Hence
$\alpha\sim Z$ is equivalent to $|Q\theta_\alpha-A|\le q^{-(D-s+2)}.$ It follows that a supporting subscheme of length $s$ gives
\[\max(|Q|,q^{D+2}|Q\theta_\alpha-A|)\le q^s.\]
Conversely, if this inequality holds, then $\deg Q\le s$, and taking
the part at $\infty$ to have length $s-\deg Q$ produces a degree $s$
subscheme through which $\alpha$ factors. Therefore, after cancelling
common factors,
\[q^{\deg(\alpha)}=\min_{\substack{A,Q\in\F_q[t]\\
Q \text{ monic, }\gcd(A,Q)=1}}\max(|Q|,q^{D+2}|Q\theta_\alpha-A|).\]
For comparison,
let $\vartheta\in\mathbb{R}/\mathbb{Z}$ be a classical frequency, and
suppose that the variables in a degree $d$ exponential sum have size
at most $B$. A reduced rational approximation $a/r$ to $\vartheta$ has
denominator $r$ and scaled error $B^d|r\vartheta-a|$. The corresponding
classical quantity is
\[\min_{\substack{a\in\mathbb{Z},\ r\ge 1\\\gcd(a,r)=1}}
\max(r,B^d|r\vartheta-a|).\]
Replacing the maximum by a sum changes this quantity by at most a factor of $2$, and the resulting expression appears in \cite[Theorem 4.1]{Vaughan1997}. Since $B=q^e$ and $q^{D+2}=q^{c+2}B^d$, this agrees, up to the fixed factor $q^{c+2}$, with the function field expression above.

The multiplication kernel also has a direct classical analogue. Define
\[K_e(\theta_\alpha)=\{h\in P_e\colon |h\theta_\alpha-A|\le q^{-(D-e+2)}\text{ for some }A\in\F_q[t]\}.\]
For $0\le j\le D-e$, the coefficient of $t^{-j-1}$ in
$h\theta_\alpha-A$ is $\alpha(ht^j)$. Consequently,
\[K_e(\theta_\alpha)=\ker C_\alpha\] and 
\[r(\alpha)=\operatorname{codim}_{P_e}K_e(\theta_\alpha).\]For comparison, let $\vartheta\in\R/\Z$ and define
\[K_B(\vartheta)=\{h\in\Z\colon|h|\le B,\ 
\|h\vartheta\|\le B^{1-d}\}.\]
The relevant geometric series estimate in the final step of Weyl
differencing is
\[\left|\sum_{|y|\le B^{d-1}}e(\vartheta hy)\right|\ll\min(B^{d-1},\|h\vartheta\|^{-1})\]
(see \cite[Lemmas 2.2--2.4]{Vaughan1997}). Thus
$K_B(\vartheta)$ is the range in which this estimate gives no saving
over the trivial bound. 

Over $\F_q(t)$, these approximation conditions are exact
linear equations, so $K_e(\theta_\alpha)$ is a vector space and
\[q^{r(\alpha)}=\frac{\#P_e}{\#K_e(\theta_\alpha)}.\]
So multiplication rank can be thought of as the function field counterpart of
the reciprocal density of $K_B(\vartheta)$. Over the integers,
$K_B(\vartheta)$ need not have any linear structure, so its cardinality
replaces the dimension of a kernel.
\end{rem}


Finally, we record the complete sum estimate that we will use, which improves upon the one used by Sawin. The cohomological vanishing is due to Katz, while the degree-by-degree Betti number estimates are due to Wan and Zhang.
\begin{thm}
\label{thm:Katz}Let $N\ge 1$, $\psi$ be a non-trivial additive character of $\F_q$, and $P\in \F_q[x_1,\ldots, x_N]$ have degree $d\ge 2$ with $p\nmid d$. Write $P_d$ for the homogeneous part of $P$ of degree $d$, and let $\Sing(P_d)$ denote the closed subscheme cut out by the first partial derivatives of $P_d$. Set $h=\codim_{\A^N}\Sing(P_d)$.

If $q>(d-1)^2$, then \[\left|\sum_{x\in\F_q^N}\psi(P(\vec x))\right|\le \frac{(d-1)^h}{1-(d-1)/\sqrt q}q^{N-h/2}.\]
\end{thm}
\begin{proof}
    Let $\mathcal{L}_\psi$ be the Artin--Schreier sheaf associated with $\psi$. The Grothendieck--Lefschetz trace formula gives \[\sum_{\vec x\in\F_q^N}\psi(P(\vec x))=\sum_i(-1)^i\Tr(\Fr_q \mid H_c^i(\A_{\overline{\F}_q}^N,P^*\mathcal{L}_\psi)).\] We first recall the range in which these cohomology groups vanish. Let $\overline{P}$ be the degree $d$ homogenization of $P$ so that \[\ol{P}(X_0,\ldots, X_N)\coloneqq X_0^dP(X_1/X_0,\ldots, X_N/X_0).\] Consider the pencil \[\mathcal{Y}\coloneqq \{\ol{P}(X_0,\ldots, X_N)=\lambda X_0^d\}\subset \P^N\times \A^1.\] and its projection $\pi\colon \mathcal{Y}\to \A^1.$ Note that $\mathcal{Y}\cap \{X_0\ne 0\}$ is the graph of $P\colon \A^N\to \A^1$, and its complement is $\{X_0=P_d(X_1,\ldots, X_N)=0\}\times \A^1$. Moreover, $\pi$ restricted to this complement is simply the projection onto $\A^1$. Then, since $R\Gamma_c(\A^1,\mathcal{L}_\psi)=0$, the excision sequence and the K\"unneth formula give us 
    \begin{equation}\label{eq:complinfty}
        R\Gamma_c(\mathcal{Y},\pi^*\mathcal{L}_\psi)\cong R\Gamma_c(\A_{\overline{\F}_q}^N,P^*\mathcal{L}_\psi).
    \end{equation} Note that the singular locus of $P_d=0$ in $\P^{N-1}$ has dimension $\delta = N-h-1$. 
    To apply Katz's results, take $X=\P^N$, $L=X_0$, $H=\overline P$, and $G=X_0^{d-1}$ in his notation.  Then his family $\widetilde f$ is precisely $\pi\colon \mathcal Y\to\A^1,$ with $\overline P=\lambda X_0^d.$ Moreover, $X\cap\{L=0\}\cap\{H=0\}=\{P_d=0\}\subset\P^{N-1},$ whose singular locus has dimension $\delta$.  On the other hand, $X\cap\{L=0\}=\P^{N-1}$ is smooth, so the quantity denoted by $\varepsilon$ in \cite{katz1999singular} is $-1$.  In particular, $\varepsilon\leq\delta$. 
    
    It follows from \cite[Corollary 22]{katz1999singular} that $R^b\pi_*\overline{\Q}_\ell$ is tame at infinity for every $b\geq N+\delta$. Since $\pi$ is proper, the compactly supported Leray spectral sequence and the projection formula give \[E_2^{a,b}=H_c^a(\A^1_{\overline{\F}_q},\mathcal L_\psi\otimes R^b\pi_*\overline{\Q}_\ell)\Rightarrow H_c^{a+b}(
   \mathcal Y_{\overline{\F}_q},
   \pi^*\mathcal L_\psi).\] Using \cite[Corollary 14(1)]{katz1999singular} therefore gives \[E_2^{a,b}=0\text{ for }a+b\ge N+\delta+2.\]Consequently,\[H_c^i(
   \mathcal Y_{\overline{\F}_q},\pi^*\mathcal L_\psi)=0 \text{ for }i\ge N+\delta+2.\]Since $N+\delta+2=2N-h+1,$ combining this with \eqref{eq:complinfty} gives\[H_c^i(\A^N_{\overline{\F}_q},P^*\mathcal L_\psi)=0\text{ for }i>2N-h.\] On the other hand, note that $P^*\mathcal{L}_\psi[N]$ is perverse on $\A^N$, so Artin vanishing gives vanishing for $i<N$, i.e. these cohomology groups are non-zero only for \[N\le i \le 2N-h.\] Now, for $0\le j\le N$, \cite[Corollary 5.3.3]{WanZhang} and its proof gives \[\dim H_c^{N+j}(\A^N_{\overline{\F}_q},P^*\mathcal L_\psi)\le (d-1)^{N-j}.\]Finally, Deligne's theory of weights gives \[\left|\Tr(\Fr_q \mid H_c^{N+j}(\A_{\overline{\F}_q}^N,P^*\mathcal{L}_\psi))\right|\le (d-1)^{N-j}q^{(N+j)/2}.\] Combining this with the trace formula equality at the beginning and the cohomological non-vanishing range, as well as setting $j'=N-h-j$, we obtain 
   \begin{align*}
       \left|\sum_{x\in\F_q^N}\psi(P(\vec x))\right|&\le \sum_{j=0}^{N-h}(d-1)^{N-j}q^{(N+j)/2}\\&=(d-1)^hq^{N-h/2}\sum_{j'=0}^{N-h}\left(\frac{d-1}{\sqrt q}\right)^{j'}\\&\le \frac{(d-1)^h}{1-(d-1)/\sqrt q}q^{N-h/2},
   \end{align*} as desired.
\end{proof}

\section{Preliminary lemmas}\label{sec:prelim}

Let $R=k[x_{1},\ldots,x_{n}]$ only in this section and give each
$x_{i}$ an integer weight $w_{i}$. Then, for any polynomial $f$, denote
by $\init_{w}(f)$ the sum of monomial terms of maximal weight. Recall
that the initial ideal (with respect to these weights) of an ideal
$I\subset R$ is the ideal 
\[
\init_{w}(I)=(\init_{w}(f)\colon f\in I).
\]

\begin{lem}
\label{lem:initialideal}Let $I\subset R$ be a homogeneous ideal
with respect to the usual grading on $R$. Then, 
\[
\height(I)=\height(\init_{w}(I)).
\]
\end{lem}

\begin{proof}
This is \cite[Theorem 15.17]{eisenbud1995}: the point is that
$R/I$ and $R/\init_{w}(I)$ are respectively the generic and special
fibers of a suitable flat family over $\A^{1}$ and hence have the same
dimension.
\end{proof}

The multiplication matrices considered below have a special shifted form.
The following elementary observation will allow us to locate enough of their
pivot degrees.

\begin{lem}
\label{lem:pivotintervals}Let $E\ge1$ and $M\ge0$, and let
\[A=(b_{i+j})_{0\le i\le E-1,\,0\le j\le M}\]
be a matrix over any field. Perform Gaussian elimination on its row space
using the columns in the order $M,M-1,\ldots,0$. Then, the set of pivot
columns is a union of at most two intervals of integers.
\end{lem}

\begin{proof}
Reverse both the rows and the columns. The columns of the resulting matrix
have the form
\[D_j=(y_j,y_{j+1},\ldots,y_{j+E-1})^{\mathsf T},\] with $0\le j\le M$ and
where $y_j=b_{M+E-1-j}$,
and we now scan them from left to right. If the first $E$ columns are
independent, or if the matrix ends before the first dependence, the pivot
columns form one interval. Otherwise, let $0\le\rho<E$ be the first index
for which $D_\rho$ is dependent on the preceding columns. Equivalently, we may write
\[
D_\rho+\sum_{a=0}^{\rho-1}q_aD_a=0 \text{ and }z_i=y_{i+\rho}+\sum_{a=0}^{\rho-1}q_ay_{i+a}.
\]
Thus $z_0=\cdots=z_{E-1}=0$.

We first observe that the leading $\rho\times\rho$ matrix
$(y_{i+j})_{0\le i,j<\rho}$ is nonsingular. Indeed, suppose that
a non-zero vector $(c_0,\ldots,c_{\rho-1})$ lies in its kernel and set
\[v_i=\sum_{j=0}^{\rho-1}c_jy_{i+j}.\]
Then $v_0=\cdots=v_{\rho-1}=0$, while we also have
\[v_{i+\rho}+\sum_{a=0}^{\rho-1}q_av_{i+a}=\sum_{j=0}^{\rho-1}c_jz_{i+j}=0\]
for $0\le i\le E-1-\rho$. This recurrence applied successively implies that $v_0=\cdots=v_{E-1}=0$, contradicting the linear independence of
$D_0,\ldots,D_{\rho-1}$. 

For every $h\ge0$ for which the columns are defined, we have
\[R_h=D_{\rho+h}+\sum_{a=0}^{\rho-1}q_aD_{a+h}
=(z_h,z_{h+1},\ldots,z_{h+E-1})^{\mathsf T}.\]
If the vector on the right-hand side is zero for every remaining $h$, all subsequent
columns are dependent and the pivots form one interval. Otherwise, let $\ell\ge 1$ be the smallest integer for which
$z_{E-1+\ell}\ne 0$. For $0\le h<\ell$, we have $R_h=0$, so that $D_{\rho +h}$ is a linear combination of $D_{a+h}$ for $0\le a\le \rho-1$. 
Since every such column $D_{a+h}$ occurs before $D_{\rho+h}$, none of
the columns
\[D_\rho,\ldots,D_{\rho+\ell-1}\]
introduces a new pivot.

Now let $0\le j\le\min(E-\rho-1,M-\rho-\ell).$ By minimality of $\ell$, note that the first non-zero coordinate of
\[R_{\ell+j}=(z_{\ell+j},\ldots,z_{\ell+j+E-1})^{\mathsf T}\]
occurs in row $E-1-j$, where it equals $z_{E-1+\ell}$. This means that the
first non-zero coordinates of $R_\ell,R_{\ell+1},$ and so on,
occur successively in rows $E-1,E-2,\ldots,\rho$. In particular,
these vectors are linearly independent.

Let $U$ be the span of the first $\rho$ columns. Since the leading $\rho\times \rho$ matrix is nonsingular (as we showed earlier), the projection of the first $\rho$ coordinates is injective on $U$. As we just checked, the $R_{\ell+j}$ is zero in those coordinates, which means that there can be no non-zero linear combination of the $R_{\ell+j}$ that live in $U$.

Finally, in the definition
\[R_{\ell+j}=D_{\rho+\ell+j}+\sum_{a=0}^{\rho-1}q_aD_{a+\ell+j},\]
every column in the sum has index smaller than
$\rho+\ell+j$. So it follows inductively that each of $D_{\rho+\ell},D_{\rho+\ell+1},$ and so on,
introduces a new pivot, until either the matrix has full row rank or
there are no columns left. Consequently, the pivot columns form at
most two intervals.
\end{proof}

\section{Bounding the singular locus}\label{sec:boundingsing}

Write 
\[F(\vec{x})=\sum_{m=0}^{c}t^{m}F_{m}(\vec{x}).\]
Note that $F_{m}$ doesn't have the variable $t$ and is in fact a
homogeneous polynomial of degree $d$ in $\F_{q}[x_{1},\ldots,x_{n}]$.
We write $\partial_{i}$ as shorthand for $\partial/\partial x_{i}$.
Note that we have 
\[(\partial_{i}F)(\vec{x})=\sum_{m=0}^{c}t^{m}(\partial_{i}F_{m})(\vec{x}).\]
We write 
\[
P_{\alpha,F}(\vec{g})=\alpha(F(\vec{g}(t))),
\]
which is a homogeneous polynomial of degree $d$ on $P_{e}^{n}$ (viewed
as an affine space). Write 
\[
g_{i}(t)=\sum_{u=0}^{e}a_{u,i}t^{u},\vec{a}_{u}=(a_{u,1},\ldots,a_{u,n}),\text{ and }\vec{g}(t)=\sum_{u=0}^{e}\vec{a}_{u}t^{u}.
\]
If $P\in k[z_{1},\ldots,z_{N}]$ is homogeneous of degree $d$, define
\[
\Sing(P)=V\left(\frac{\partial P}{\partial z_{1}},\ldots,\frac{\partial P}{\partial z_{N}}\right)\subset\A^{N}.
\]
Thus $\Sing P$ is the affine singular cone. In particular, $\Sing P_{\alpha,F}$
is cut out by the partial derivatives $\partial P_{\alpha,F}/\partial a_{u,i}$
for $0\le u\le e,1\le i\le n$ inside $P_{e}^{n}$. 

The rest of this section is dedicated to proving the following.
\begin{prop}
\label{prop:codimbound}Assume $V(F)\subset\P^{n-1}_{\F_q(t)}$ is smooth. Then
for any $\alpha\in P_{de+c}^{\vee}$, put
\[
\mu(\alpha)=\max\left(\left\lfloor
\frac{\max(r(\alpha)-c,0)}{d-1}\right\rfloor-1,0\right).
\]
Then,
\[
\codim\Sing P_{\alpha,F}\ge n\mu(\alpha).
\]
\end{prop}

\begin{proof}
Let $\Jac_{\alpha,F}$ denote the Jacobian ideal of $P_{\alpha,F}$, generated by the partial derivatives
\[\frac{\partial P_{\alpha,F}}{\partial a_{u,i}} \text{ such that }0\le u\le e, 1\le i\le n. \] Thus, \[ \codim\Sing P_{\alpha,F}=\height\Jac_{\alpha,F}. \]
The overall strategy is to arrange that $F_0$ is smooth and then use the special form of the multiplication matrix to find many pivot degrees divisible by $d-1$. A weighted degeneration turns each of these pivots into a copy of the gradient ideal of $F_0$ in a separate block of variables.

\medskip
\noindent\emph{1. Reducing to a smooth constant fiber.}
We may first extend the ground field to its algebraic closure without any loss of generality. Consider the family\[ \mathcal X=V(F)\subset\P^{n-1}\times\A^1. \] Consider the relative singular locus \[ \mathcal Z= V(F, \partial_1F,\ldots,\partial_nF) \subset\P^{n-1}\times\A^1. \] The fiber above a point $b$ is simply the singular locus $V(F)$ with $t$ evaluated at $b$. By properness of $\mathcal Z\to\A^1$, its image is closed. Moreover, its generic fiber is empty because $V(F)\subset\P^{n-1}$ is smooth. We can therefore choose $b\in\ol{\F_q}$ such that the fiber is smooth. Now, consider the translation \[ T_bh(t)=h(t+b). \] Observe that the map $T_b$ gives an invertible linear transformation of every space $P_a$. Let $F'=T_bF$ and $\alpha'=\alpha\circ T_b^{-1}$ so that \[ P_{\alpha',F'}(T_b\vec g)=P_{\alpha,F}(\vec g). \] Moreover, for $h\in P_e$ and $u\in P_{(d-1)e+c}$, we have \[ C_{\alpha'}(T_bh)(T_bu) =\alpha'(T_b(hu)) =\alpha(hu). \] Thus translation preserves both $r(\alpha)$ and $\codim\Sing P_{\alpha,F}$. Hence, replacing $(F,\alpha)$ by $(F',\alpha')$ allows us to assume that $F_0$ is smooth.

\medskip \noindent\emph{2. Jacobian equations with selected pivots.}
Let $M=(d-1)e+c$ and $L_{\alpha}=\im C_{\alpha}\subset P_M^{\vee}$.
In the monomial bases, the matrix of $C_\alpha$ is
\[
(\alpha(t^{u+v}))_{0\le u\le e,\,0\le v\le M}.
\]
Perform Gaussian elimination on $L_\alpha$ using the columns
$[t^M],[t^{M-1}],\ldots,[t^0]$. By Lemma \ref{lem:pivotintervals},
the $r(\alpha)$ pivot degrees form a union of at most two intervals.

At most $c$ pivot degrees are greater than $(d-1)e$. Therefore, at least
$\max(r(\alpha)-c,0)$ pivots lie in $[0,(d-1)e]$, still in at most two
intervals. An interval containing $s$ consecutive integers contains at
least $\lfloor s/(d-1)\rfloor$ multiples of $d-1$. For two intervals
of sizes $s_1$ and $s_2$, we also have
\[
\left\lfloor\frac{s_1}{d-1}\right\rfloor+
\left\lfloor\frac{s_2}{d-1}\right\rfloor
\ge \left\lfloor\frac{s_1+s_2}{d-1}\right\rfloor-1.
\]
It follows that at
least $\mu(\alpha)$ of these pivot degrees are divisible by $d-1$.

Let $w_j=(d-1)q_j$ be these selected pivot degrees. The corresponding
rows in reduced echelon form have the shape
\[
\ell_j=\sum_{w=0}^{w_j}c_{j,w}[t^w]
\] with $c_{j,w_j}\ne0$.
In particular, $\ell_j$ belongs to the span of
$[t^0],\ldots,[t^{(d-1)e}]$.
Let 
\[
G_{j,i}=\ell_{j}((\partial_{i}F)(\vec{g}(t))).
\]
We first check that $G_{j,i}$ lies in the Jacobian ideal cutting
out $\Sing P_{\alpha,F}$. Indeed, by the chain rule, we have 
\[
\frac{\partial P_{\alpha,F}}{\partial a_{u,i}}=\alpha\left(\sum_{r=1}^{n}(\partial_{r}F)(\vec{g}(t))\frac{\partial g_{r}(t)}{\partial a_{u,i}}\right)=\alpha(t^{u}(\partial_{i}F)(\vec{g}(t))).
\]
So if $h_{j}(t)=\sum_{u=0}^eh_{j,u}t^u$ is a lift of $\ell_{j}$, i.e. so that $C_{\alpha}(h_{j})=\ell_{j}$,
then we have \[G_{j,i}=\alpha(h_{j}(\partial_{i}F)(\vec{g}(t)))=\sum_{u=0}^e h_{j,u}\frac{\partial P_{\alpha,F}}{\partial a_{u,i}},\]
a linear combination of the coordinate partial derivatives defining
the Jacobian ideal. In particular, every $G_{j,i}$ is in $\Jac_{\alpha,F}$.

\medskip \noindent\emph{3. Computing the weighted initial forms.} Next, we write $G_{j,i}$ in terms of the symmetric polarizations
of $\partial_{i}F_{m}$. Let $B_{m,i}$ be the symmetric multilinear
polarization of $\partial_{i}F_{m}$, i.e. such that 
\[
B_{m,i}(\vec{a},\ldots,\vec{a})=(\partial_{i}F_{m})(\vec{a})
\]
(here we use that $p>d-1$). We can write 
\[
(\partial_{i}F_{m})(\vec{g}(t))=\sum_{0\le u_{1},\ldots,u_{d-1}\le e}B_{m,i}(\vec{a}_{u_{1}},\ldots,\vec{a}_{u_{d-1}})t^{u_{1}+\cdots+u_{d-1}},
\]
so 
\[
[t^{w}](\partial_{i}F)(\vec{g}(t))=\sum_{\substack{0\le m\le c, \\ 0\le w-m\le(d-1)e}}\sum_{\substack{0\le u_{1},\ldots,u_{d-1}\le e, \\\sum_k u_{k}=w-m}}B_{m,i}(\vec{a}_{u_{1}},\ldots,\vec{a}_{u_{d-1}}).
\]
Hence, we have 
\[
G_{j,i}=\sum_{w=0}^{w_{j}}c_{j,w}\sum_{\substack{0\le m\le c, \\ 0\le w-m\le(d-1)e}}\sum_{\substack{0\le u_{1},\ldots,u_{d-1}\le e, \\\sum_k u_{k}=w-m}}B_{m,i}(\vec{a}_{u_{1}},\ldots,\vec{a}_{u_{d-1}}).
\]
 Next, we choose a suitable weight on the monomials of $\ol{\F_{q}}[a_{u,i}]_{0\le u\le e,1\le i\le n}$:
 
Let $\Lambda$ be any integer greater than $(d-1)e^{2}$ and give each
$a_{u,i}$ the weight
\[
e^{2}+\Lambda u-(e-u)^{2}-u^{2}.
\]
For convenience, we suppress this weight from the notation and write $\init$ for
$\init_w$. Now, each monomial in $B_{m,i}(\vec a_{u_1},\ldots, \vec a_{u_{d-1}})$ with $\sum_{k=1}^{d-1}u_k=w-m$ has weight\[
(d-1)e^2+\Lambda(w-m)
-\sum_{k=1}^{d-1}\bigl((e-u_k)^2+u_k^2\bigr).
\]
The final sum lies between $0$ and $(d-1)e^2$, inclusive. Since
$\Lambda>(d-1)e^2$, the terms of largest weight in $G_{j,i}$ must have
$w=w_j$ and $m=0$ (also recall that the Gaussian elimination made every $w\le w_j$ in the $j$th row).

Recall that $w_j=(d-1)q_j$. Subject to
$u_1+\cdots+u_{d-1}=(d-1)q_j,$
the sum
\[
\sum_{k=1}^{d-1}\bigl((e-u_k)^2+u_k^2\bigr)
\]
is minimized uniquely when $u_1=\cdots=u_{d-1}=q_j$. Indeed, replacing
a pair $(u_a,u_b)$ with $(u_a-1,u_b+1)$ decreases this sum whenever
$u_a\ge u_b+2$. It now follows that
\[
\init(G_{j,i})=c_{j,w_j}(\partial_iF_0)(\vec a_{q_j}).
\]

\medskip \noindent\emph{4. The height estimate.} For each $j$, let $I_j\subset \ol{\F_q}[a_{u,i}]_{0\le u\le e,1\le i\le n}$ be the ideal generated by $\init(G_{j,i})$ for $1\le i\le n$. Smoothness of $F_0$ and Euler's identity give
$V(\partial_1F_0,\ldots,\partial_nF_0)=\{0\}\subset\A^n,$ so $\height I_j=n$. 

Distinct selected pivots give distinct $q_j$, and
hence the ideals $I_j$ use disjoint blocks of variables. Their sum
therefore has height at least $n\mu(\alpha)$.

Since each $c_{j,w_j}$ is non-zero, the preceding initial-form
calculation shows that
\[
\sum_j I_j\subset\init(\Jac_{\alpha,F}).
\]
Lemma \ref{lem:initialideal} now gives
\[
\codim\Sing P_{\alpha,F}
=\height\Jac_{\alpha,F}
=\height\init(\Jac_{\alpha,F})
\ge n\mu(\alpha),
\]
as desired.
\end{proof}

\section{Computing the singular series}\label{sec: compute_sing_series}
For the remainder of the paper, we assume the hypotheses of
Theorem \ref{thm:main}. In particular, $V(F)$ is smooth,
$\eta>0$, and hence $q>(d-1)^2$.

Let $Z\subset\P^{1}$ be a finite subscheme of length $r$. We call
a functional 
\[
\beta\colon H^{0}(Z,\O_{Z}((de+c)\infty))\to\F_{q}
\]
\textit{primitive on $Z$} as above. Define 
\[
S_{Z}(\beta)\coloneqq q^{-nr}\sum_{\vec{x}\in H^{0}(Z,\O_Z(e\infty))^{n}}\psi(\beta(F(\vec{x})))
\]
and 
\[
\mathfrak{S}(F)\coloneqq\sum_{Z\ge0}\sum_{\beta\text{ primitive on }Z}S_{Z}(\beta).
\]
Here $Z$ ranges over effective divisors defined over $\F_q$. We regard the zero functional as primitive on $Z=0$ and set $S_0(0)=1$.
\begin{lem}\label{lem:expsumprim}
Let $Z$ have length $r\ge1$ and $\beta$ be primitive on $Z$. Then,
putting
\[
\mu_r=\max\left(\left\lfloor
\frac{\max(r-c,0)}{d-1}\right\rfloor-1,0\right),
\]
we have
\[|S_{Z}(\beta)|\le \frac{1}{1-(d-1)/\sqrt q}\left(\frac{d-1}{\sqrt q}\right)^{n\mu_r}.\]
\begin{proof}
We construct an $\alpha\in P_{d(r-1)+c}^{\vee}$ such that
$S_Z(\beta)=q^{-nr}S(\alpha)$, where $S(\alpha)$ is formed using $r-1$ as the degree parameter in the notation of the preceding sections. 

For any line bundle $L$, we have an induced map $F\colon L^{n}\to L^{\otimes d}\otimes\O(c\infty).$
Since all line bundles on $Z$ are trivial, choose an isomorphism
$\varphi\colon\O_{Z}((r-1)\infty)\to\O_{Z}(e\infty)$. This induces
an isomorphism $\Phi_{d}\coloneqq \varphi^{\otimes d} \otimes \Id_{\O_Z(c\infty)}\colon \O_{Z}((d(r-1)+c)\infty)\to\O_{Z}((de+c)\infty),$ so that $\Phi_d(F(\vec{x}))=F(\varphi(\vec{x})).$
Let 
\[
\rho_{m}\colon H^{0}(\P^{1},\O((m(r-1)+c)\infty))\to H^{0}(Z,\O_{Z}((m(r-1)+c)\infty))
\]
be the restriction, and note that $\rho_{d}$ and $\rho_{d-1}$ are
surjective by Riemann--Roch. Also, let 
\[
\phi_{1}\colon H^{0}(\P^{1},\O((r-1)\infty))\to H^{0}(Z,\O_{Z}((r-1)\infty))
\]
be the restriction, which is an isomorphism, again by Riemann--Roch.
Then define $\alpha$ to be $\beta\circ H^{0}(Z,\Phi_{d})\circ\rho_{d}$.
Since $\phi_{1}$ is an isomorphism and $\psi(\beta(F(\vec{x})))=\psi(\alpha(F(\vec{x})))$
by construction, we have $S_{Z}(\beta)=q^{-nr}S(\alpha)$. 

Next, we claim that $r(\alpha)=r$. The argument is essentially the
same as that of Lemma \ref{lem:rankvsdeg}, namely that we can write
$C_{\alpha}$ as $\rho_{d-1}^{\vee}\circ B_{Z}\circ\phi_{1}$, where
$B_{Z}\colon H^{0}(Z,\O_{Z}((r-1)\infty))\to H^{0}(Z,\O_{Z}(((d-1)(r-1)+c)\infty))^{\vee}$
is induced by the pairing defined using $\beta$: \[B_Z(a)(u)=\beta(\Phi_d(au)).\] Now, $\rho_{d-1}^{\vee}$
and $\phi_{1}$ are of full rank by the previous paragraph, and $B_{Z}$
is of full rank by the assumption of primitivity. 

If $\mu_r=0$, the claimed inequality follows from the trivial bound
$|S_Z(\beta)|\le1$. Suppose that $\mu_r>0$. Proposition \ref{prop:codimbound}, applied with degree parameter $r-1$, gives \[\codim\Sing P_{\alpha,F}\ge n\mu_r>0.\]
In particular, $P_{\alpha,F}$ is non-zero, so Theorem \ref{thm:Katz}
applies and gives the desired estimate.
\end{proof}
\end{lem}

There are at most $\#\Sym^{r}(\P^{1})(\F_{q})\le q^{r}/(1-q^{-1})$ effective divisors of degree $r$, and for each such divisor there are at most $q^r$ primitive functionals. Since
\[
\mu_r\ge\frac{r-c}{d-1}-2,
\] Lemma \ref{lem:expsumprim} then gives that the degree $r$ contribution is bounded above by
\begin{equation}
\frac{q^{n(1/2-\log_q(d-1))(c/(d-1)+2)}}{(1-q^{-1})(1-(d-1)/\sqrt q)}q^{-(\eta + d - 2)r}.\label{eq:degreeestimate}
\end{equation}
Since $\eta+d-2>0$, it follows that $\mathfrak{S}(F)$ converges
absolutely. 

Next, if we decompose $Z$ as $\sum_{v\in|\P^{1}|}m_{v}[v]$, we can
use the Chinese remainder theorem to decompose each
sum $S_{Z}(\beta)$ into a product of local contributions at each
place:
\[
q^{-R(n-1)\deg v}\#\{\vec{x}\bmod\pi_{v}^{R}\colon F_{v}(\vec{x})=0\}=\sum_{m=0}^{R}\sum_{\beta\text{ primitive on }m[v]}S_{m[v]}(\beta).
\]
Indeed, put $A_R=\O_{\P^1,v}/\pi_v^R$, so that
$|A_R|=q^{R\deg v}$. After using the standard local trivializations
(and the equation $F_\infty(u,\vec{x})=u^cF(u^{-1},\vec{x})$ at infinity),
every $\F_q$-linear functional on $A_R$ has a unique least conductor
$m\le R$ and is the pullback of a primitive functional on $m[v]$.
Orthogonality of characters therefore gives\[\sum_{m=0}^{R}\sum_{\beta\text{ primitive on }m[v]}S_{m[v]}(\beta)
=|A_R|^{-n}\sum_{\vec{x}\in A_R^n}\sum_{\lambda\in A_R^\vee}
\psi(\lambda(F_v(\vec{x})))
=|A_R|^{1-n}\#\{\vec{x}\in A_R^n\colon F_v(\vec{x})=0\},\]
which is the displayed identity. The Chinese remainder theorem also gives
\[
S_Z(\beta)=\prod_v S_{m_v[v]}(\beta_v),
\]and $\beta$ is primitive
on $Z$ precisely when every $\beta_v$ is primitive on $m_v[v]$.
The absolute convergence proved above now permits passage to the limit
in $R$ and regrouping by places. It follows that the local density
limits exist and that
\[
\mathfrak{S}(F)=\prod_{v\in|\P^{1}|}\ell_{v}(F).
\]
To compare this completed support series with the original character sum,
apply the bijection of Lemma \ref{lem:minimalsupport}. Moreover, the
restriction map $P_e\to H^0(Z,\O_Z(e\infty))$ is surjective with fibers
of cardinality $q^{e+1-\deg Z}$, so
\[
S(\alpha)=q^{n(e+1)}S_{Z}(\beta).
\]

\section{Putting everything together}
\begin{proof}[Proof of Theorem \ref{thm:main}]
By the previous section, we have 
\begin{align*}
q^{(de+c+1)-n(e+1)}N_{e}(F) & =q^{-n(e+1)}\sum_{\deg\alpha\le e}S(\alpha)+q^{-n(e+1)}\sum_{\deg\alpha>e}S(\alpha)\\
 & =\mathfrak{S}(F)-\sum_{\substack{\deg Z>e, \\ \beta\text{ primitive on }Z}}S_{Z}(\beta)+q^{-n(e+1)}\sum_{\deg\alpha>e}S(\alpha)\\
 & =\prod_{v\in|\P^{1}|}\ell_{v}(F)-\sum_{\substack{\deg Z>e, \\ \beta\text{ primitive on }Z}}S_{Z}(\beta)+q^{-n(e+1)}\sum_{\deg\alpha>e}S(\alpha).
\end{align*}
Put
\[
\mu_e=\max\left(\left\lfloor
\frac{\max(e+1-c,0)}{d-1}\right\rfloor-1,0\right).
\]
For all sufficiently large $e$, we have $\mu_e>0$. If $\deg\alpha>e$,
Lemma \ref{lem:rankvsdeg} gives $r(\alpha)=e+1$, and Proposition
\ref{prop:codimbound} gives $\codim\Sing P_{\alpha,F}\ge n\mu_e>0,$ so $P_{\alpha,F}$ is non-zero. Moreover, $\eta>0$ implies that $q>(d-1)^2$. Then, applying Theorem \ref{thm:Katz}, the preceding codimension bound gives
the minor arc contribution is bounded as follows:
\[\left|q^{-n(e+1)}\sum_{\deg\alpha>e}S(\alpha)\right| \le q^{-n(e+1)}\sum_{\deg\alpha>e}\left|S(\alpha)\right|
 \le\frac{q^{de+c+1}}{1-(d-1)/\sqrt q}
 \left(\frac{d-1}{\sqrt q}\right)^{n\mu_e}.\]
Since
\[
\mu_e\ge\frac{e+1-c}{d-1}-2,
\]
the minor arc contribution is $O_{F,q,d,n}(q^{-\eta e})$.

Next, by \eqref{eq:degreeestimate}, we have 
\begin{align*}
\left|\sum_{\deg Z>e,\beta\text{ primitive on }Z}S_{Z}(\beta)\right| & \le\sum_{\deg Z>e,\beta\text{ primitive on }Z}\left|S_{Z}(\beta)\right|\\
 & \le\sum_{i=0}^{\infty}\frac{q^{n(1/2-\log_q(d-1))(c/(d-1)+2)}}{(1-q^{-1})(1-(d-1)/\sqrt q)}
 q^{-(\eta+d-2)(e+1)}q^{-(\eta+d-2)i}\\
 & =\frac{q^{n(1/2-\log_q(d-1))(c/(d-1)+2)}}{(1-q^{-1})(1-(d-1)/\sqrt q)(1-q^{-(\eta+d-2)})}q^{-(\eta+d-2)(e+1)}.
\end{align*}
Since $d\ge2$, this is again $O_{F,q,d,n}(q^{-\eta e})$. Combining these contributions
gives 
\[
q^{(de+c+1)-n(e+1)}N_{e}(F)=\prod_{v\in|\P^{1}|}\ell_{v}(F)+O_{F,q,d,n}(q^{-\eta e}),
\]
which after rearranging is precisely Theorem \ref{thm:main}.
\end{proof}
\bibliographystyle{plain}
\bibliography{Quadratic_Birch}
\end{document}